\documentclass[11pt]{article}
\usepackage[colorlinks]{hyperref}
\usepackage{color}
\usepackage{graphicx}
\usepackage{graphics}
\usepackage{makeidx}
\usepackage{showidx}
\usepackage{latexsym}
\usepackage{amssymb}
\usepackage{verbatim}
\usepackage{amsmath}
\usepackage{amsthm}
\usepackage{amsfonts}

\newtheorem{Theorem}{Theorem}[section]

\newtheorem{cor}{Corollary}[section]
\newtheorem{lemma}{Lemma}

\title{Foliations of Tangent Bundle in a Finsler Manifold}
\author{H. Attarchi and M. M. Rezaii}

\begin{document}
\maketitle
\noindent
\begin{abstract}
In this paper, a frame is introduced on tangent bundle of a Finsler manifold in a manner that it makes some simplicity to study the properties of the natural foliations in tangent bundle. Moreover, we show that the indicatrix bundle of a Finsler manifold with lifted sasaki metric and natural almost complex structure on tangent bundle cannot be a sasakian manifold.
\vspace{.3cm}

\end{abstract}

%%%%%%%%%%%%%%%%%%%%%%%%%%%%%%%%%%%%%%%%%%%%%%%%%%%%%%%%%%%%%%%%%%%%%%%%%%%%%%%%%%%%%%%%%%%%%%%%%%%%%%%%%%%%%
\section{Introduction}

First time Sasaki~\cite{sasaki}, construct a natural Riemannian metric $G$ on the tangent bundle $TM$ of the Riemannian manifold $(M,g)$. Then $G$ was called the Sasaki metric on $TM$ and it was the main tool in studying interrelations between the geometries of $(M,g)$ and $(TM,G)$. Later, this idea was used to construct a Riemannian metric on tangent bundle of a Finsler manifold. The geometric objects that occur in Finsler geometry depend on both point and direction, therefore, the tangent bundle of a Finsler manifold plays a major role in the study of Finslerian objects. To emphasize this, there are several studies of interrelations between the geometry of foliations on the tangent bundle of a Finsler manifold and the geometry of the Finsler manifold itself~\cite{bejancu2008,bejancu}. Among the natural foliations of tangent bundle on a Finsler manifold, indicatrix bundle and Liouville vector fields play more important roles.

Let $(M,F)$ be a $n$-dimensional Finsler manifold and $TM$ its tangent bundle with sasaki lifted metric $G$. In this paper to study the geometry of some foliations of $TM$ which they were presented in~\cite{bejancu}, a new local frame of vector fields in $TTM$ is considered. There are some difficulties of working with natural foliations of tangent bundle of a Finsler manifold with respect to the natural basis $\partial/\partial y^i$ and $\delta/\delta x^j$. For example, there is no explicit decomposition of tangent vector fields of indicatrix bundle by vector fields $\partial/\partial y^i$ and $\delta/\delta x^j$. In Sections 3,4 and 5, there are some theorems and results that show the advantages of using the novel frame on tangent bundle of a Finsler manifold in some cases.

Aiming at our purpose, this paper is organized in the following way. In section 2, a short review of Finsler manifolds~\cite{bao-chern-shen,bejancu b1} is done and the notations which is needed in the followings are presented. In Section 3, a frame of local vector fields in $\Gamma TTM$ is introduced to study the indicatrix bundle and other natural foliations of tangent bundle of a Finsler manifold. Moreover, local components of Levi-Civita connection of sasaki metric $G$ on $TM$ are calculated in this basis. In~\cite{bejancu}, six natural foliations of tangent bundle of a Finsler manifold is introduced and some properties of them such as totally geodesic and bundle-like with respect to metric $G$ are studied. In Section 4, by using the new frame introduced in Section 3 some more theorem about these foliations are proved. Finally in Section 5, it is proved that the indicatrix bundle with its contact structure given in~\cite{bejancu2008} cannot be a Sasakian manifold~\cite{Blair}. Therefore, we should not have any expectation of properties of Sasakian manifolds in Riemannian geometry on indicatrix bundle as a Riemannian submanifold of $TM$ with restricted metric $\bar{G}$ of sasaki metric $G$.

%%%%%%%%%%%%%%%%%%%%%%%%%%%%%%%%%%%%%%%%%%%%%%%%%%%%%%%%%%%%%%%%%%%%%%%%%%%%%%%%%%%%%%%%%%%%%%%%%%%%%%%%%%%%%
\section{Preliminaries and Notations}
Let $(M,F)$ be an $n$-dimensional smooth Finsler manifold and $TM$ be its tangent bundle. If $(x^i)$ be the local coordinate on $M$ then the local coordinate on $TM$ is shown by $(x^i,y^i)$ where $(y^i)$ are the fibre coordinate. With respect to local coordinate system induced on $TM$, the natural local frame fields on $TM$ are given by $\frac{\partial}{\partial y^i}$ and $\frac{\partial}{\partial x^i}$. The vertical distribution $VTM$ is locally spanned by $\{\frac{\partial}{\partial y^1},\ldots,\frac{\partial}{\partial y^n}\}$. Considering the fundamental function $F$, then the horizontal distribution $HTM$ as a complementary distribution of $VTM$ is naturally defined as follows. The spray coefficients $G^i$ of fundamental function $F$ are given by:
$$G^i:=\frac{g^{ij}}{4}\left(\frac{\partial^2F^2}{\partial y^j\partial x^k}y^k-\frac{\partial F^2}{\partial x^j}\right)$$
where $(g^{ij})$ is the inverse matrix of Hessian matrix $F$ given as follows: $$g:=(g_{ij})=\left(\frac{1}{2}\frac{\partial^2F^2}{\partial y^i\partial y^j}\right)$$
Moreover, the nonlinear connection coefficients $G_i^j$ of $F$ are defined by:
$$G_i^j=\frac{\partial G^j}{\partial y^i}.$$
The horizontal distribution $HTM$ with respect to these nonlinear connections is given by:
$$HTM=<\frac{\delta}{\delta x^1},\ldots,\frac{\delta}{\delta x^n}>$$
where $\frac{\delta}{\delta x^i}=\frac{\partial}{\partial x^i}-G_i^j\frac{\partial}{\partial y^j}$. The lie brackets of these bases are given as follows:
$$[\frac{\delta}{\delta x^i},\frac{\delta}{\delta x^j}]=R_{\ ij}^k\frac{\partial}{\partial y^k},\ \ \ [\frac{\delta}{\delta x^i},\frac{\partial}{\partial y^j}]=G_{ij}^k\frac{\partial}{\partial y^k}$$
where $R_{\ ij}^k=\frac{\delta G_i^k}{\delta x^j}-\frac{\delta G_j^k}{\delta x^i}$ and $G_{ij}^k=\frac{\partial G_i^k}{\partial y^j}$. The dual local 1-forms of $\frac{\delta}{\delta x^i}$ and $\frac{\partial}{\partial y^i}$ are denoted by $dx^i$ and $\delta y^i$, respectively, where $\delta y^i:=dy^i+G_j^idx^j$. Then, the lifted sasaki metric $G$ on $TM$ in these local frames is given as follows:
$$G:=g_{ij}dx^i\otimes dx^j+g_{ij}\delta y^i\otimes\delta y^j$$
which it is a Riemannian metric on $TM$. The natural almost complex structure on $TTM$ which is compatible with metric $G$ is defined by:
\begin{equation}~\label{a.c.s.}
J:=\frac{\delta}{\delta x^i}\otimes\delta y^i-\frac{\partial}{\partial y^i}\otimes dx^i
\end{equation}
The \emph{indicatrix bundle} $I\!M$ of Finsler manifold $(M,F)$ is defined by:
$$I\!M:=\{(x,y)\in TM | F(x,y)=1\}$$
It is proved in~\cite{bejancu2008} that $I\!M$ with $(\varphi,\eta,\xi,\bar{G})$ is a contact metric manifold, where
\begin{equation}~\label{contact1}
\left\{
\begin{array}{l}
\eta:=y^ig_{ij}dx^j,\ \ \ \xi:=y^i\frac{\delta}{\delta x^i}\\ \cr
\varphi:=J|_D, \ \ \ \varphi(\xi):=0\\ \cr
D:=\{X\in TTM | \eta(X)=\eta(JX)=0\}
\end{array}
\right.
\end{equation}
and $\bar{G}$ is restriction of sasaki metric $G$ to indicatrix bundle. Distribution $D$ defined in~(\ref{contact1}) is called \emph{contact distribution} of indicatrix bundle.

In addition, throughout the paper the Einstein convention, that is, repeated indices with one upper index and one lower index denote summation over their range is used. The indices $i,j,k,...$ are used for range $1,\ldots,n$ if not stated otherwise.

%%%%%%%%%%%%%%%%%%%%%%%%%%%%%%%%%%%%%%%%%%%%%%%%%%%%%%%%%%%%%%%%%%%%%%%%%%%%%%%%%%%%%%%%%%%%%%%%%%%%%%%%%%%%%
\section{A Frame on Indicatrix Bundle of a Finsler Manifold}
Supposed $(M,F)$ be an $n$-dimensional Finsler manifold. In Riemannian manifold $(TM,G)$, the orthogonal distribution to vertical Liouville vector field $L=y^i\frac{\partial}{\partial y^i}$ in vertical distribution $VTM$ is denoted by $V'TM$. By definition of $V'TM$, it is easy to see that $V'TM$ is a foliation in $TTM$. Therefore, there is the local chart $(U,\varphi)$ on $TM$ such that
$$TTM|_U=<\frac{\bar{\partial}}{\bar{\partial}y^1},...,\frac{\bar{\partial}}{\bar{\partial}y^{n-1}},e_1,...,e_{n+1}>$$
where
\begin{equation}~\label{vtm}
V'TM|_U=<\frac{\bar{\partial}}{\bar{\partial}y^1},...,\frac{\bar{\partial}}{\bar{\partial}y^{n-1}}>
\end{equation}
It is obvious that $\frac{\bar{\partial}}{\bar{\partial}y^a}$ for all $a=1,...,n-1$ are vector fields in $VTM$ and they can be written as a linear combination of the natural basis $\{\frac{\partial}{\partial y^1},...,\frac{\partial}{\partial y^n}\}$ of $VTM$ as follows:
$$\frac{\bar{\partial}}{\bar{\partial}y^a}=E_a^i\frac{\partial}{\partial y^i}\ \ \ \forall a=1,...,n-1$$
where $E_a^i$ is the $(n-1)\times n$ matrix of maximum rank. The first property of this matrix is $E_a^ig_{ij}y^j=0$
achieved by the feature $G(\frac{\bar{\partial}}{\bar{\partial}y^a},L)=0$. Moreover, on $U$ the vertical distribution $VTM$ is locally spanned by:
$$\{\frac{\bar{\partial}}{\bar{\partial}y^1},...,\frac{\bar{\partial}}{\bar{\partial}y^{n-1}},L\}$$
Then, by using the natural almost complex structure $J$ presented in~(\ref{a.c.s.}), the new local basis in $TTM$ is introduced as follows:
\begin{equation}~\label{new bas}
TTM|_U=<\frac{\bar{\delta}}{\bar{\delta}x^a},\ \xi,\ \frac{\bar{\partial}}{\bar{\partial}y^a},\ L>
\end{equation}
where $\frac{\bar{\delta}}{\bar{\delta}x^a}:=J\frac{\bar{\partial}}{\bar{\partial}y^a}$. By considering the vector frame~(\ref{new bas}) on $TTM$ it is possible to locally decomposed every sections of contact distribution $D$ or every tangent vector field of $IM$ to basic vector fields $\frac{\bar{\delta}}{\bar{\delta}x^a}$, $\frac{\bar{\partial}}{\bar{\partial}y^a}$ and $\xi$. Equivalently, it means that:
$$TIM|_U=<\frac{\bar{\delta}}{\bar{\delta}x^a},\xi,\frac{\bar{\partial}}{\bar{\partial}y^a}> \ ,\ \ \ \ D|_U=<\frac{\bar{\delta}}{\bar{\delta}x^a},\frac{\bar{\partial}}{\bar{\partial}y^a}>.$$
The Sasakian metric $G$ on $TM$ can be shown in the new frame field~(\ref{new bas}) as follows:
\begin{equation}~\label{met.mat.}
G:=\left(%
\begin{array}{cccc}
  g_{ab} & 0 & 0 & 0\\
  0 & F^2 & 0 & 0\\
  0 & 0 & g_{ab} & 0\\
  0 & 0 & 0 & F^2\\
\end{array}%
\right)
\end{equation}
where $a,b\in\{1,...,n-1\}$ and $g_{ab}=G(\frac{\bar{\delta}}{\bar{\delta}x^a},\frac{\bar{\delta}}{\bar{\delta}x^b})
=G(\frac{\bar{\partial}}{\bar{\partial} y^a},\frac{\bar{\partial}}{\bar{\partial}y^b})=g_{ij}E_a^iE_b^j$. Furthermore, Lie brackets of vector fields~(\ref{new bas}) are presented as follows:
\begin{equation}~\label{lie bracket}
\left\{
\begin{array}{l}
(1) \ [\frac{\bar{\delta}}{\bar{\delta}x^a},\frac{\bar{\delta}}{\bar{\delta}x^b}]=(\frac{\bar{\delta}E_b^i}{\bar{\delta}x^a}
-\frac{\bar{\delta}E_a^i}{\bar{\delta}x^b})\frac{\delta}{\delta x^i}+E_a^iE_b^jR_{ij}^k\frac{\partial}{\partial y^k},\\
(2) \ [\frac{\bar{\delta}}{\bar{\delta}x^a},\frac{\bar{\partial}}{\bar{\partial}y^b}]=(\frac{\bar{\delta}E_b^k}{\bar{\delta}x^a}
+E_a^iE_b^jG_{ij}^k)\frac{\partial}{\partial y^k}-\frac{\bar{\partial}E_a^i}{\bar{\partial}y^b}\frac{\delta}{\delta x^i},\\
(3) \
[\frac{\bar{\partial}}{\bar{\partial}y^a},\frac{\bar{\partial}}{\bar{\partial}y^b}]=(
\frac{\bar{\partial}E_b^i}{\bar{\partial}y^a}-\frac{\bar{\partial}E_a^i}{\bar{\partial}y^b})\frac{\partial}{\partial y^i},\\
(4) \
[\frac{\bar{\delta}}{\bar{\delta}x^a},\xi]=-(E_a^iG_i^j+\xi(E_a^j))\frac{\delta}{\delta x^j}+E_a^iy^jR_{ij}^k
\frac{\partial}{\partial y^k},\\
(5) \
[\frac{\bar{\partial}}{\bar{\partial}y^a},\xi]=\frac{\bar{\delta}}{\bar{\delta}x^a}-(\xi(E_a^j)+E_a^iG_i^j)
\frac{\partial}{\partial y^j},\\
(6) \
[\frac{\bar{\delta}}{\bar{\delta}x^a},L]=-L(E_a^i)\frac{\delta}{\delta x^i},\\
(7) \
[\frac{\bar{\partial}}{\bar{\partial}y^a},L]=\frac{\bar{\partial}}{\bar{\partial}y^a}-L(E_a^i)
\frac{\partial}{\partial y^i},\\
(8) \
[\xi,\xi]=[L,L]=[\xi,L]+\xi=0.
\end{array}
\right.
\end{equation}
Consider local vector fields $\frac{\bar{\delta}}{\bar{\delta}x^a}$, $\frac{\bar{\partial}}{\bar{\partial}y^a}$ and $\xi$ and sasaki metric $G$~(\ref{met.mat.}) then the local components of Levi-Civita connection $\nabla$ given by:
\begin{equation}~\label{levi-civita}
\left\{\begin{array}{l}
2G(\nabla_XY,Z)=XG(Y,Z)+YG(X,Z)-ZG(X,Y)\cr
-G([X,Z],Y)-G([Y,Z],X)+G([X,Y],Z)
\end{array}
\right.
\end{equation}
are as follows:
\begin{equation}~\label{levicivita1}
\left\{
\begin{array}{l}
\nabla_{\frac{\bar{\delta}}{\bar{\delta}x^a}}\frac{\bar{\delta}}{\bar{\delta}x^b}=(\Gamma_{ab}^e+
\frac{\bar{\delta}E_b^j}{\bar{\delta}x^a}E_d^kg_{jk}g^{de})
{\frac{\bar{\delta}}{\bar{\delta}x^e}}+(-g_{ab}^e+\frac{1}{2}R_{ab}^e)\frac{\bar{\partial}}{\bar{\partial}y^e}
+\frac{1}{2F^2}\bar{R}_{ab}\xi,\\
\nabla_{\frac{\bar{\delta}}{\bar{\delta}x^a}}\frac{\bar{\partial}}{\bar{\partial}y^b}
=\left(\frac{1}{2}E_b^jE_d^kE_a^i(\frac{\delta g_{jk}}{\delta x^i}-G_{ik}^hg_{hj}+G_{ij}^hg_{hk})+\frac{\bar{\delta}E_b^j}{\bar{\delta}x^a}
E_d^kg_{jk}\right)g^{de}\frac{\bar{\partial}}{\bar{\partial}y^e}\cr +(g_{ab}^e-\frac{1}{2}R_{bad}g^{de})\frac{\bar{\delta}}{\bar{\delta}x^e}
+\frac{1}{2F^2}R_{ab}\xi,\\
\nabla_{\frac{\bar{\partial}}{\bar{\partial}y^b}}\frac{\bar{\delta}}{\bar{\delta}x^a}=(g_{ab}^e-
\frac{1}{2}R_{bad}g^{de}+\frac{\bar{\partial}E_a^i}{\bar{\partial}y^b}E_d^kg_{ik}g^{de}
)\frac{\bar{\delta}}{\bar{\delta}x^e}+\frac{1}{F^2}(\frac{1}{2}R_{ab}-g_{ab})\xi\cr
+\frac{1}{2}E_a^iE_b^jE_d^k(\frac{\delta g_{jk}}{\delta x^i}-G_{ik}^hg_{hj}
-G_{ij}^hg_{hk})g^{de}\frac{\bar{\partial}}{\bar{\partial}y^e},\\
\nabla_{\frac{\bar{\partial}}{\bar{\partial}y^a}}\frac{\bar{\partial}}{\bar{\partial}y^b}=\frac{1}{2}E_a^iE_b^jE_d^k
(G_{ik}^hg_{hj}+G_{jk}^hg_{hi}-\frac{\delta g_{ij}}{\delta x^k})g^{de}\frac{\bar{\delta}}{\bar{\delta}x^e}\cr +(g_{ab}^e+\frac{\bar{\partial}E_b^j}{\bar{\partial}y^a}E_d^kg_{kj}g^{de})\frac{\bar{\partial}}{\bar{\partial}y^e}
-\frac{1}{F^2}g_{ab}L,\\
\nabla_{\frac{\bar{\delta}}{\bar{\delta}x^a}}\xi=\frac{1}{2}\bar{R}_{da}g^{de}
\frac{\bar{\delta}}{\bar{\delta}x^e}-\frac{1}{2}R_{ad}g^{de}\frac{\bar{\partial}}{\bar{\partial}y^e},\\
\nabla_{\xi}\frac{\bar{\delta}}{\bar{\delta}x^a}=(\xi(E_a^i)E_d^kg_{ik}+E_a^iG_i^hg_{hk}E_d^k+\frac{1}{2}\bar{R}_{da})
g^{de}\frac{\bar{\delta}}{\bar{\delta}x^e}+\frac{1}{2}R_{ad}g^{de}\frac{\bar{\partial}}{\bar{\partial}y^e},\\
\nabla_{\frac{\bar{\partial}}{\bar{\partial}y^a}}\xi=(\delta_a^e-\frac{1}{2}R_{ad}g^{de})
\frac{\bar{\delta}}{\bar{\delta}x^e},\\
\nabla_{\xi}\frac{\bar{\partial}}{\bar{\partial}y^a}=-\frac{1}{2}R_{ad}g^{de}\frac{\bar{\delta}}{\bar{\delta}x^e}
+(\xi(E_a^i)g_{ik}+E_a^iG_i^jg_{jk})E_d^kg^{de}\frac{\bar{\partial}}{\bar{\partial}y^e},\\
\nabla_{\frac{\bar{\delta}}{\bar{\delta}x^a}}L=\nabla_L\frac{\bar{\delta}}{\bar{\delta}x^a}-L(E_a^i)E_d^kg_{ik}g^{de}
\frac{\bar{\delta}}{\bar{\delta}x^e}=0,\\
\nabla_{\frac{\bar{\partial}}{\bar{\partial}y^a}}L
-\frac{\bar{\partial}}{\bar{\partial}y^a}=\nabla_L\frac{\bar{\partial}}{\bar{\partial}y^a}-L(E_a^i)E_d^kg_{ik}g^{de}
\frac{\bar{\partial}}{\bar{\partial}y^e}=0,\\
\nabla_{\xi}\xi=\nabla_{\xi}L=\nabla_L\xi-\xi=\nabla_LL-L=0.
\end{array}
\right.
\end{equation}
where
\begin{equation}~\label{symb.}
g_{ab}^c=g_{abd}g^{dc}=\frac{1}{2}E_a^iE_b^jE_d^kg_{ijk}g^{dc}\ ,\ \ \ \ \ \ \Gamma_{ab}^c=E_a^iE_b^jE_d^k\Gamma_{ij}^hg_{hk}g^{dc}
\end{equation}
\centerline{$R_{ab}^c=R_{dab}g^{dc}=E_a^iE_b^jE_d^kR_{ij}^hg_{hk}g^{dc},$}
\centerline{$\bar{R}_{ab}=(\frac{\bar{\delta}E_b^i}{\bar{\delta}x^a}
-\frac{\bar{\delta}E_a^i}{\bar{\delta}x^b})g_{ij}y^j\ ,\ \ \ \ \ \ R_{ab}=E_a^iE_b^jR_{ij}$}
\noindent and, $(g^{ab})$ is the inverse matrix of $(g_{ab})$.

Suppose that the Levi-Civita connection and metric on indicatrix bundle are denoted by $\bar{\nabla}$ and $\bar{G}$, respectively, which $\bar{G}$ is the restriction of metric $G$. To compute the components of Levi-Civita connection $\bar{\nabla}$ on indicatrix bundle $I\!M$ from~(\ref{levicivita1}), the \emph{Guass Formula}~\cite{lee}:
\begin{equation}~\label{guass}
\nabla_XY=\bar{\nabla}_XY+H(X,Y)
\end{equation}
where $H$ is the \emph{second fundamental form} of $I\!M$ in $TM$ is needed. From~(\ref{levicivita1}), it is obvious that $\nabla$ is tangent to $I\!M$ for all combinations of $\frac{\bar{\delta}}{\bar{\delta}x^a},\ \xi$ and $\frac{\bar{\partial}}{\bar{\partial}y^a}$ except $\nabla_{\frac{\bar{\partial}}{\bar{\partial}y^a}}\frac{\bar{\partial}}{\bar{\partial}y^b}$. Therefore, $\bar{\nabla}$ is equal to $\nabla$ for the other combinations of $\frac{\bar{\delta}}{\bar{\delta}x^a},\ \xi$ and $\frac{\bar{\partial}}{\bar{\partial}y^a}$ by using the Gauss formula~(\ref{guass}).

The curvature tensor $R$ of $\nabla$ defined by $R(X,Y)Z=\nabla_X\nabla_YZ-\nabla_Y\nabla_XZ-\nabla_{[X,Y]}Z$
is related to the curvature tensor $\bar{R}$ of $\bar{\nabla}$ as follows:
\begin{equation}~\label{cur}
\left\{
\begin{array}{l}
R(\frac{\bar{\delta}}{\bar{\delta}x^a},\frac{\bar{\delta}}{\bar{\delta}x^b})\frac{\bar{\partial}}{\bar{\partial}y^c}
=\bar{R}(\frac{\bar{\delta}}{\bar{\delta}x^a},\frac{\bar{\delta}}{\bar{\delta}x^b})
\frac{\bar{\partial}}{\bar{\partial}y^c}+\frac{1}{F^2}R_{cab}L\\
R(\frac{\bar{\delta}}{\bar{\delta}x^a},\frac{\bar{\partial}}{\bar{\partial}y^b})\frac{\bar{\delta}}{\bar{\delta}x^c}
=\bar{R}(\frac{\bar{\delta}}{\bar{\delta}x^a},\frac{\bar{\partial}}{\bar{\partial}y^b})
\frac{\bar{\delta}}{\bar{\delta}x^c}+\frac{1}{2F^2}(R_{bac}-2g_{abc})L\\
R(\frac{\bar{\partial}}{\bar{\partial}y^a},\frac{\bar{\partial}}{\bar{\partial}y^b})
\frac{\bar{\partial}}{\bar{\partial}y^c}
=\bar{R}(\frac{\bar{\partial}}{\bar{\partial}y^a},\frac{\bar{\partial}}{\bar{\partial}y^b})
\frac{\bar{\partial}}{\bar{\partial}y^c}-\frac{1}{F^2}g_{bc}\frac{\bar{\partial}}{\bar{\partial}y^a}+
\frac{1}{F^2}g_{ac}\frac{\bar{\partial}}{\bar{\partial}y^b}\\
R(\frac{\bar{\delta}}{\bar{\delta}x^a},\frac{\bar{\partial}}{\bar{\partial}y^b})
\frac{\bar{\partial}}{\bar{\partial}y^c}
=\bar{R}(\frac{\bar{\delta}}{\bar{\delta}x^a},\frac{\bar{\partial}}{\bar{\partial}y^b})
\frac{\bar{\partial}}{\bar{\partial}y^c}+\frac{1}{2}E_c^iE_b^jE_a^k
(G_{ik}^hg_{hj}+G_{jk}^hg_{hi}-\frac{\delta g_{ij}}{\delta x^k})L\\
R(\frac{\bar{\delta}}{\bar{\delta}x^a},\frac{\bar{\partial}}{\bar{\partial}y^b})\xi=
\bar{R}(\frac{\bar{\delta}}{\bar{\delta}x^a},\frac{\bar{\partial}}{\bar{\partial}y^b})\xi-\frac{1}{2F^2}R_{ab}L\\
R(\frac{\bar{\partial}}{\bar{\partial}y^a},\xi)\frac{\bar{\delta}}{\bar{\delta}x^b}=
\bar{R}(\frac{\bar{\partial}}{\bar{\partial}y^a},\xi)\frac{\bar{\delta}}{\bar{\delta}x^b}-\frac{1}{2F^2}R_{ab}L\\
R(\frac{\bar{\delta}}{\bar{\delta}x^a},\xi)\frac{\bar{\partial}}{\bar{\partial}y^b}=
\bar{R}(\frac{\bar{\delta}}{\bar{\delta}x^a},\xi)\frac{\bar{\partial}}{\bar{\partial}y^b}-\frac{1}{F^2}R_{ab}L
\end{array}
\right.
\end{equation}
For the other combinations of $\frac{\bar{\delta}}{\bar{\delta}x^a},\frac{\bar{\partial}}{\bar{\partial}y^a}$ and $\xi$, tensor fields $R$ and $\bar{R}$ are equal.

%%%%%%%%%%%%%%%%%%%%%%%%%%%%%%%%%%%%%%%%%%%%%%%%%%%%%%%%%%%%%%%%%%%%%%%%%%%%%%%%%%%%%%%%%%%%%%%%%%%%%%%%%%%%%
\section{Foliations on $(TM,G)$}
In this section, the local frame~(\ref{new bas}) on $TM$ is used to study some properties of natural foliations on tangent bundle of a Finsler manifold. In~\cite{bejancu}, a comprehensive study was done on six foliations of $TM$ presented as follows:
\begin{enumerate}
\item{$L$: vertical Liouville vector field.}
\item{$\xi$: horizontal Liouville vector field.}
\item{$L\oplus\xi$}.
\item{$VTM$: defined by $VTM=<\frac{\partial}{\partial y^1},\ldots,\frac{\partial}{\partial y^n}>$.}
\item{$V'TM$: defined in~(\ref{vtm}).}
\item{$V^{\perp}TM$: which is perpendicular to $L$ in $TTM$ with respect to the metric $G$.}
\end{enumerate}
A. Bejancu, in~\cite{bejancu}, studied the properties of these foliations such as totally geodesic and bundle-like
for sasaki metric $G$. Here, some more theorems are proved of these foliations about these properties by help of the frame which it was introduced in Section 3.
\begin{cor}~\label{bun lik 1}
The sasaki lifted metric $G$ is bundle-like for foliation $V^{\perp}TM$.
\end{cor}
\begin{proof}
By help of~(\ref{levicivita1}), it is obtained that:
$$G(\nabla_LL+\nabla_LL,X)=0 \ \ \ \ \forall X\in \Gamma(V^{\perp}TM)$$
and this completes the proof.
\end{proof}
\begin{Theorem}~\label{bun lik 2}
The metric $G$ is bundle-like for foliation $V'TM$ if and only if $(M,g)$ is a Riemannian manifold.
\end{Theorem}
\begin{proof}
From~(\ref{levicivita1}), it is a straightforward calculation to obtain:
$$G(\nabla_XY+\nabla_YX,\frac{\bar{\partial}}{\bar{\partial}y^c})=0\ \ \ \forall c\in\{1,\ldots,n-1\}, X,Y\in\Gamma(HTM\oplus L)$$
except $G(\nabla_{\frac{\bar{\delta}}{\bar{\delta}x^a}}\frac{\bar{\delta}}{\bar{\delta}x^b}
+\nabla_{\frac{\bar{\delta}}{\bar{\delta}x^b}}\frac{\bar{\delta}}{\bar{\delta}x^a},
\frac{\bar{\partial}}{\bar{\partial}y^b})$ which is equal to $-2g_{abc}$. Therefore, $G$ is bundle-like for foliation $V'TM$ if and only if $g_{abc}=0$, and it leads to $g_{ijk}=0$ by definition $g_{abc}$ in~(\ref{symb.}). This completes the proof.
\end{proof}
\begin{Theorem}~\label{tot1}
The foliations $V'TM$ and $V^{\perp}TM$ are not totally geodesic with respect to the Levi-Civita connection of Riemannian metric $G$ on $TM$.
\end{Theorem}
\begin{proof}
From~(\ref{levicivita1}), it is obtained that:
$H(\frac{\bar{\partial}}{\bar{\partial}y^a},\frac{\bar{\partial}}{\bar{\partial}y^b})=-\frac{1}{F^2}g_{ab}L$
for both foliations $V'TM$ and $V^{\perp}TM$, which it cannot be vanish. This completes the proof.
\end{proof}
\begin{cor}~\label{bun lik 3}
The metric $G$ for foliations $L$ and $L\oplus\xi$ cannot be bundle-like.
\end{cor}
\begin{proof}
From~(\ref{levicivita1}), it is obtained that:
$$G(\nabla_{\frac{\bar{\partial}}{\bar{\partial}y^a}}\frac{\bar{\partial}}{\bar{\partial}y^b}
+\nabla_{\frac{\bar{\partial}}{\bar{\partial}y^b}}\frac{\bar{\partial}}{\bar{\partial}y^a},L)=-2g_{ab}$$
Therefore, $g_{ab}=0$ is a necessary condition to $G$ be bundle-like for $L$ or $L\oplus\xi$, which it is impossible.
\end{proof}

%%%%%%%%%%%%%%%%%%%%%%%%%%%%%%%%%%%%%%%%%%%%%%%%%%%%%%%%%%%%%%%%%%%%%%%%%%%%%%%%%%%%%%%%%%%%%%%%%%%%%%%%%%%%%
\section{Sasakian Structure and Indicatrix Bundle of a Finsler Manifold}
Now, let $(\bar{M},\bar{\varphi},\bar{\eta},\bar{\xi},\bar{g})$ be a contact metric manifold defined in~\cite{Blair}. In~\cite{bejancu2}, the new connection $\tilde{\nabla}$ was presented on the contact metric manifold $\bar{M}$ as follows:
$$\tilde{\nabla}_XY=\nabla_XY-\bar{\eta}(X)\nabla_Y\bar{\xi}-\bar{\eta}(Y)\nabla_X\bar{\xi}+\bar{g}(X,\bar{\varphi} Y)\bar{\xi}+\frac{1}{2}(\mathcal{L}_{\bar{\xi}}\bar{g})(X,Y)\bar{\xi}$$
where $\nabla$ is Levi-Civita connection of Riemannian metric $\bar{g}$. It was proved in~\cite{bejancu2} that the contact metric manifold $\bar{M}$ is a Sasakain manifold if and only if
\begin{equation}~\label{sasa1}
(\tilde{\nabla}_X\bar{\varphi})Y=0 \ \ \ \ \forall X,Y\in\Gamma(T\bar{M})
\end{equation}
Since the indicatrix bundle has the contact metric structure in Finslerian manifolds by Proposition 4.1 in~\cite{bejancu2008}, here it is tried to find an answer to the question that "\emph{Can the indicatrix bundle with contact structure given in~(\ref{contact1}) be a Sasakian manifold?}". First, the following Lemma is proved in order to reduce the number of calculations. To answer this question, the followings are needed:
\begin{lemma}\label{redu}
If $(\bar{M},\bar{\varphi},\bar{\eta},\bar{\xi},\bar{g})$ be a contact metric manifold with contact distribution $\bar{D}$, then $\bar{M}$ is Sasakian manifold if and only if:
$$(\tilde{\nabla}_X\bar{\varphi})Y=0 \ \ \ \ \forall X,Y\in\Gamma(\bar{D})$$
\end{lemma}
\begin{proof}
For all $\bar{X}\in\Gamma(T\bar{M})$, they can be written in the form $X+f\bar{\xi}$ where $X\in\Gamma\bar{D}$ and $f\in C^{\infty}(\bar{M})$. Therefore:
$$(\tilde{\nabla}_{\bar{X}}\bar{\varphi})\bar{Y}=(\tilde{\nabla}_{X+f\bar{\xi}}\bar{\varphi})(Y+h\bar{\xi})=
(\tilde{\nabla}_X\bar{\varphi})Y+(\tilde{\nabla}_{f\bar{\xi}}\bar{\varphi})Y+(\tilde{\nabla}_X\bar{\varphi})h\bar{\xi}$$
$$+(\tilde{\nabla}_{f\bar{\xi}}\bar{\varphi})h\bar{\xi}=(\tilde{\nabla}_X\bar{\varphi})Y
+f(\tilde{\nabla}_{\bar{\xi}}\bar{\varphi}Y-\bar{\varphi}\tilde{\nabla}_{\bar{\xi}}Y)+\tilde{\nabla}_X\bar{\varphi}
(h\bar{\xi})-\bar{\varphi}(\tilde{\nabla}_Xh\bar{\xi})$$
$$+f(\tilde{\nabla}_{\bar{\xi}}\bar{\varphi}{h\bar{\xi}}-\bar{\varphi}\tilde{\nabla}_{\bar{\xi}}h\bar{\xi})
=(\tilde{\nabla}_X\bar{\varphi})Y$$
\noindent The lemma is proved using Theorem 3.2 in~\cite{bejancu2} and the last equation.
\end{proof}
Now, the following Theorem can be expressed:
\begin{Theorem}~\label{sasaki}
Let $(M,F)$ be a Finsler manifold. Then, indicatrix bundle $I\!M$ with its natural contact structure given in~(\ref{contact1}) can never be a Sasakian manifold.
\end{Theorem}
\begin{proof}
From lemma~\ref{redu} and local frame~(\ref{new bas}), $I\!M$ is a Sasakian manifold if and only if:
$$(\tilde{\nabla}_{\frac{\bar{\delta}}{\bar{\delta}x^a}}\varphi)\frac{\bar{\delta}}{\bar{\delta}x^b}= (\tilde{\nabla}_{\frac{\bar{\delta}}{\bar{\delta}x^a}}\varphi)\frac{\bar{\partial}}{\bar{\partial}y^b}= (\tilde{\nabla}_{\frac{\bar{\partial}}{\bar{\partial}y^a}}\varphi)\frac{\bar{\delta}}{\bar{\delta}x^b}= (\tilde{\nabla}_{\frac{\bar{\partial}}{\bar{\partial}y^e}}\varphi)\frac{\bar{\partial}}{\bar{\partial}y^b}=0$$
\noindent Using~(\ref{levicivita1}), one of the components in above equations where it must be zero is $g_{ab}$
which it is impossible and shows that the indicatrix bundle cannot have a Sasakian structure on contact structure given in~(\ref{contact1}).
\end{proof}

%%%%%%%%%%%%%%%%%%%%%%%%%%%%%%%%%%%%%%%%%%%%%%%%%%%%%%%%%%%%%%%%%%%%%%%%%%%%%%%%%%%%%%%%%%%%%%%%%%%%%%%%%%%%%
\noindent\textbf{Another proof for Theorem~\ref{sasaki}}\\
The following argument was presented in Chapter 6 of~\cite{Blair}. We consider $TM=I\!M\times\mathbb{R}$ for the Finslerian manifold $(M,F)$ and introduce the almost complex structure $\bar{J}$ by means of $\varphi$ defined in~(\ref{contact1}) as follows:
$$\bar{J}(X+fL)=\varphi(X)-f\xi+\eta(X)L$$
where $X$ is a vector field tangent to indicatrix bundle and $\xi,\eta$ were defined in Section 2. Using a straight calculation, it can be seen that $\bar{J}$ is equal to $J$ defined in~(\ref{a.c.s.}). The contact structure $(\varphi,\eta,\xi)$ will be a Sasakian structure if and only if $(\varphi,\eta,\xi)$ is normal, that is, $\bar{J}$ (or $J$) is integrable. Integrability of $\bar{J}$ (or $J$) is equivalent to vanishing the following equations
$$N_J(\frac{\delta}{\delta x^i},\frac{\delta}{\delta x^j})=-N_J(\frac{\partial}{\partial y^i},\frac{\partial}{\partial y^j})=-R_{ij^k}\frac{\partial}{\partial y^k}$$
$$N_J(\frac{\delta}{\delta x^i},\frac{\partial}{\partial y^j})=-R_{ij}^k\frac{\delta}{\delta x^k}$$
Therefore, $\bar{J}$ (or $J$) is integrable if and only if $M$ is a flat manifold. Up to now, it can be shown that $I\!M$ is Sasakian if and only if $M$ is flat. Furthermore, it was proved that $\xi$ is a killing vector field if $I\!M$ be a Sasakian manifold~\cite{Blair}. Therefore, $M$ has constant curvature 1 using  Theorem 3.4 in~\cite{bejancu} and it is a contradiction to the previous result which shows that $M$ is flat if $I\!M$ is a Sasakian manifold. $\Box$

%%%%%%%%%%%%%%%%%%%%%%%%%%%%%%%%%%%%%%%%%%%%%%%%%%%%%%%%%%%%%%%%%%%%%%%%%%%%%%%%%%%%%%%%%%%%%%%%%%%%%%%%%%%%%

\vspace{.5cm}
%%%%%%%%%%%%%%%%%%%%%%%%%%%%%%%%%%%%%%%%%%%%%%%%%%%%%%%%%%%%%%%%%%%%%%%%%%%%%%%%%%%%%%%%%%%%%%%%%%%%%%%%%%%%%

\noindent Hassan Attarchi;\\
e-mail: hassan.attarchi@aut.ac.ir\\
Ph.D. Student in\\
Department of Mathematics and Computer Science,\\
Amirkabir University of Technology, Tehran, Iran.
\vspace{.3cm}

\noindent Corresponding author: Morteza Mir Mohammad Rezaii;\\
e-mail: mmreza@aut.ac.ir\\
Associated Professor in\\
Department of Mathematics and Computer Science,\\
Amirkabir University of Technology, Tehran, Iran.

%%%%%%%%%%%%%%%%%%%%%%%%%%%%%%%%%%%%%%%%%%%%%%%%%%%%%%%%%%%%%%%%%%%%%%%%%%%%%%%%%%%%%%%%%%%%%%%%%%%%%%%%%%%%%

\begin{thebibliography}{1}

\bibitem{bao-chern-shen} D. Bao, S.S. Chen and Z. Shen, An Introduction to Riemann-Finsler Geometry, Springer,
New York, 2000.

\bibitem{bejancu b1} A. Bejancu and H. R. Farran, Geometry of Pseudo-Finsler Submanifolds, Kluwer Acad. Publ.,
Dordrecht, 2000.

\bibitem{bejancu2008} A. Bejancu, Tangent Bundle and Indicatrix Bundle of a Finsler Manifold,
Journal of Kodai Mathematics, 31 (2008), 272-306.

\bibitem{bejancu2} A. Bejancu, Kahler contact distribution,
Journal Geometry and Physics, 60 (2010), 1958-1967.

\bibitem{bejancu} A. Bejancu and H. R. Farran, Finsler Geometry and Natural Folitions on the Tangent Bundle,
Reports on Mathematical Physics, 58 (2006), 131-146.

\bibitem{Blair} D. E. Blair, Riemannian Geometry of Contact and Symplectic Manifolds,
Birkhauser, Boston, 2002.

\bibitem{lee} J.M. Lee, Riemannian Manifolds: An Introduction to Curvature,
Springer, New York, 1997.

\bibitem{matsumoto} M. Matsumoto, Foundations of Finsler geometry and special Finsler spaces,
Kaiseisha Press, Saikawa Otsu, 1986.

\bibitem{sasaki} S. Sasaki, On the differential geometry of tangent bundles of Riemannian manifolds,
Tohoku Math. J., 10 (1958), 338-354.

\end{thebibliography}
\end{document}